\documentclass{amsart}

\newif\ifcomments
\commentstrue 

\usepackage{amsmath,amsthm,amssymb,amscd,url,enumerate}

\usepackage[utf8]{inputenc}
\usepackage{hyperref}
\usepackage{pdflscape}
\usepackage{caption}
\usepackage[nocompress, noadjust]{cite}
\usepackage{tikz}
\usetikzlibrary{decorations.markings,arrows,matrix,cd}
\usepackage{amsfonts}
\usepackage{amsmath}
\usepackage{graphicx}
\usepackage{amsthm}
\usepackage{xcolor}
\usepackage{mathrsfs}
\usepackage[margin=1in]{geometry}

\newcommand{\bsh}{\backslash}
\newcommand{\amod}{\!\! \pmod}

\newcommand{\legendre}[2]{\genfrac{(}{)}{}{}{#1}{#2}}

\newcommand{\wtE}{\widetilde{E}}
\newcommand{\wtH}{\widetilde{H}}

\DeclareMathOperator{\Emb}{Emb}
\DeclareMathOperator{\Aut}{Aut}

\theoremstyle{plain} 
\newtheorem{theorem}{Theorem} 

\newtheorem{proposition}[theorem]{Proposition}
\newtheorem{lemma}[theorem]{Lemma}
\newtheorem{corollary}[theorem]{Corollary}

\theoremstyle{definition}

\theoremstyle{remark}
\newtheorem{remark}[theorem]{Remark}

\newtheorem*{acknowledgement}{Acknowledgements}

\numberwithin{theorem}{section}

\newcommand{\CC}{\mathbb{C}}
\newcommand{\ZZ}{\mathbb{Z}}
\newcommand{\QQ}{\mathbb{Q}}
\newcommand{\FF}{\mathbb{F}}
\newcommand{\RR}{\mathbb{R}}
\newcommand{\Fp}{\mathbb{F}_p}
\newcommand{\Fpp}{\mathbb{F}_{p^2}}
\newcommand{\grj}{\mathfrak{j}}
\newcommand{\gra}{\mathfrak{a}}
\newcommand{\grP}{\mathfrak{P}}
\newcommand{\calO}{\mathcal{O}}
\newcommand{\calM}{\mathcal{M}}
\newcommand{\calT}{\mathcal{T}}
\newcommand{\calR}{\mathcal{R}}

\newcommand{\scrR}{\mathscr{R}}

\newcommand{\calH}{\mathcal{H}}
\newcommand{\Fpbar}{\overline{\mathbb{F}}_p}
\newcommand{\End}{\operatorname{End}}

\newcommand{\disc}{\operatorname{disc}}

\newcommand{\nrd}{\operatorname{nrd}}
\newcommand{\typ}{\operatorname{Typ}}
\newcommand{\opt}{\mathrm{opt}}

\newcommand{\Qbar}{\overline{\mathbb{Q}}}

\DeclareMathOperator{\Cl}{Cl}
\DeclareMathOperator{\ells}{\mathcal{E}\ell\ell}
\newcommand{\els}{\mathcal{E}\ell\ell_{/\Fpbar}^{ss}}

\DeclareMathOperator{\Pic}{Pic}

\definecolor{Bittersweet}{rgb}{1.0, 0.44, 0.37}

\ifcomments
\newcommand{\MC}[1]{\textcolor{Bittersweet}{{\sf Mingjie:} {\sl{#1}}}}
\else
\newcommand{\MC}[1]{}
\fi

\title{On $\mathbb{F}_p$-roots of the Hilbert class polynomial modulo $p$}
\author{Mingjie Chen, Jiangwei Xue}
\date{\today}



\thanks{
Mingjie Chen was supported by NSF grants DMS-1844206, DMS-1802161.
}

\address{(Chen) Department of Mathematics, University of California San Diego, 9500 Gilman Drive, La Jolla, CA 92093-0112}
\email{mic181@ucsd.edu}

\address{(Xue) Collaborative Innovation Center of Mathematics, School of Mathematics and Statistics, Wuhan University, Luojiashan, 430072, Wuhan, Hubei, P.R. China}   

\address{(Xue) Hubei Key Laboratory of Computational Science (Wuhan
   University), Wuhan, Hubei,  430072, P.R. China.}
 \email{xue\_j@whu.edu.cn}

\begin{document}

\maketitle
\begin{abstract}
    The Hilbert class polynomial $H_{\mathcal{O}}(x)\in \ZZ[x]$ attached to an order $\mathcal{O}$ in an imaginary quadratic field $K$  is the monic polynomial whose roots are precisely the distinct $\grj$-invariants of elliptic curves over $\CC$  with complex multiplication by $\mathcal{O}$. Let $p$ be a prime inert in $K$  and strictly greater than  $|\disc(\mathcal{O})|$.   We show that the number of $\Fp$-roots of $H_\mathcal{O}(x)\amod{p}$ is either zero or $|\Pic(\mathcal{O})[2]|$ by exhibiting a free and  transitive action of $\Pic(\mathcal{O})[2]$ on the set of $\Fp$-roots of $H_\mathcal{O}(x)\amod{p}$ whenever it is nonempty.  We also provide a concrete criterion for the existence of $\Fp$-roots. A similar result was first obtained by Xiao et al.~\cite{supersingular_j_invariants} and generalized much further by Li et al.~\cite{li2021factorization} (that covers the current result) with a different approach.
\end{abstract}

\section{Introduction}
Let $\mathcal{O}$ be an order in an imaginary quadratic field $K$, and $\Pic(\mathcal{O})$ be the Picard group of $\mathcal{O}$, i.e.~the group of isomorphism classes of invertible fractional $\mathcal{O}$-ideals under multiplication.  The Hilbert class polynomial $H_{\mathcal{O}}(x)$ attached to $\mathcal{O}$ is defined to be   \begin{equation}\label{eq:hilbert}
H_{\mathcal{O}}(x)=\prod_{[\mathfrak{a}]\in \Pic(\mathcal{O})}(x-\grj(\CC/\mathfrak{a})),
\end{equation}
where $[\gra]$ denotes the isomorphism class of the invertible fractional  $\calO$-ideal $\gra$, and 
$\grj(\CC/\mathfrak{a})$ stands for the $\grj$-invariant of the complex elliptic curve $\CC/\mathfrak{a}$. It is well known that $H_{\mathcal{O}}(x)$ has integral coefficients, and it is irreducible over $\QQ$ (see \cite[\S13]{coxprimes} and \cite[Chapter~10, App., p.144]{lang-ell-func}).  


Let $p\in \mathbb{N}$ be a prime number, and $\wtH_\mathcal{O}(x)\in \FF_p[x]$ be the polynomial   obtained by  reducing  $H_\mathcal{O}(x)\in \ZZ[x]$ modulo $p$. Suppose that $p$ is non-split in $K$ so that  the roots of $\wtH_\mathcal{O}(x)$ are supersingular $\grj$-invariants, which are known to  lie in $\mathbb{F}_{p^2}$. It's natural to ask how many of them are actually in $\mathbb{F}_p$. Castryck, Panny, and Vercauteren answered this question in \cite[Theorem 26]{CPV} for special cases when $p\equiv 3\amod{4}$, $K$ is of the form $\mathbb{Q}(\sqrt{-l})$ with $l$ prime, $l<(p+1)/4$ and $\mathcal{O}$ is an order containing $\sqrt{-l}$. Their method as in \cite[Section 5.2]{CPV} counts the $\Fp$-roots by constructing supersingular elliptic curves over $\Fp$. We take a different approach here by reinterpreting the $\Fp$-roots in terms of quaternion orders,  which allows us to answer the question in more generality.

Our main result is as follows.

\begin{theorem}\label{thm:main_1}
Let $K$ be an imaginary quadratic field and  $\calO$ be an order in  $K$.  Let $p$ be a prime  inert in $K$ and strictly greater than $|\disc({\mathcal{O}})|$, and  $\mathcal{H}_p$ be  set of $\Fp$-roots of 
$\wtH_\mathcal{O}(x)$.  If $\mathcal{H}_p$ is nonempty, then it admits a regular (i.e.~free and transitive) action by the $2$-torsion subgroup $\Pic(\mathcal{O})[2]\subset \Pic(\mathcal{O})$.  In particular,  the number of $\Fp$-roots of $\wtH_\mathcal{O}(x)$ is either zero or $|\Pic(\mathcal{O})[2]|$.

Moreover, $\mathcal{H}_p \neq \emptyset$ if and only if for every prime factor $\ell$ of $\disc(\calO)$,  either condition (i) or (ii) below holds for $\ell$ depending on  its parity: 
 \begin{enumerate}
     \item[(i)] $\ell \neq 2$ and the Legendre symbol $\legendre{-p}{\ell}=1$;
     \item[(ii)]  $\ell = 2$ and one of the following conditions holds:
         \subitem (a) $p \equiv 7 \amod{8}$;
         \subitem (b) $-p + \frac{\disc(\mathcal{O})}{4} \equiv 0,\,1 $ or $4 \amod{8}$;
         \subitem (c) $-p + \disc(\mathcal{O})\equiv 1 \amod{8}$.
 \end{enumerate}
\end{theorem}


The assumption that $|\disc(\calO)| < p$ immediately implies  that $p$ does
not divide the discriminant of $H_\mathcal{O}(x)$ by an influential
work of Gross and Zagier \cite{GrossZagier}.  Therefore,
$\wtH_\mathcal{O}(x)$ has no repeated roots. We provide an
alternative proof of this fact under the current assumptions in Corollary~\ref{cor:no-repeat-roots}.


\begin{remark}
  After the first of version of this manuscript appeared on the web, Jianing Li kindly informed us that a similar result to Theorem~\ref{thm:main_1} has firstly been obtained in \cite[Theorem 1.1]{supersingular_j_invariants} under the assumption that $|\disc(\calO)|<4\sqrt{p/3}$. Moreover,   Li et al.~used a method similar to \cite{supersingular_j_invariants} and generalized it much further in a joint work \cite{li2021factorization}. Their result is as follows. 
  Let $\grj_0=\grj(\CC/\calO)$, and put $L:=\QQ(\grj_0)$. If $p$ coprime to the index $[\calO_L: \ZZ[\grj_0])$ (e.g.~if $p\nmid \disc(\calO)$), then they completely determined the factorization of $\wtH_\calO(x)$ in $\Fp[x]$. Partial results are also obtained without the co-primality condition. In particular, the results of Theorem~\ref{thm:main_1} has been covered in \cite[Theorem~4.1]{li2021factorization}. On the other hand, 
  the current project was initiated in May 2021 during an online discussion between the authors.  Unaware of the significant progress made by aforementioned works, we worked independently and obtained Theorem~\ref{thm:main_1} by a completely different method: we count the $\Fp$-roots by demonstrating a regular action using quaternion orders, whereas the aforementioned works count by studying the factorization of $p$ in $L$.
\end{remark}

For the reader's convenience, we reproduce the celebrated formula of Gauss on the order of $\Pic(\mathcal{O})[2]$.


\begin{theorem}[{{\cite[Proposition 3.11]{coxprimes}}}] Let $r$ be the number of odd primes dividing $\disc(\mathcal{O})$. Define the number $\mu$ as follows: if $\disc(\mathcal{O})\equiv 1 \amod{4}$, then $\mu = r$, and if $\disc(\mathcal{O})\equiv 0 \amod{4}$, then $\disc(\mathcal{O}) = -4n$, where $n>0$, and $\mu$ is determined as follows:

\begin{center}
$\mu = 
\begin{cases}
r  \  \; &\text{if} \;  n \equiv 3 \amod{4};\\
r+1 \  \; &\text{if} \;  n \equiv 1,2 \amod{4};\\
r+1 \  \; &\text{if} \;  n \equiv 4 \amod{8};\\
r+2 \  \; &\text{if} \;  n \equiv 0 \amod{8}.
\end{cases} $
\end{center}
Then $|\Pic(\calO)[2]|=2^{\mu-1}$.
\end{theorem}


This paper is organized as follows. In section \ref{sec:Fp-roots}, we
give a reinterpretation of $\mathcal{H}_p$ in terms of quaternion orders. In section
\ref{sec:free_transitive}, we show that there is a regular action of
$\Pic(\mathcal{O})[2]$ on $\mathcal{H}_p$ whenever
$\calH_p\neq \emptyset$, and provide a nonemptiness criterion for
$\calH_p$.  
Throughout the paper, the prime $p\in \mathbb{N}$ is assumed to be non-split in $K$.  The notation $B_{p, \infty}$ is reserved for the
unique quaternion $\QQ$-algebra ramified precisely at $p$ and infinity. Given a set $X$ and an  equivalence relation on $X$, the equivalence class of an element $x\in X$ is denoted by $[x]$.


\begin{acknowledgement}
The first author would like to thank WIN5 research group members Sarah Arpin, Kristin Lauter, Renate Scheidler, Katherine Stange and Ha Tran for helpful comments on some initial ideas on this problem. The first author would also like to thank Kiran Kedlaya and Rachel Pries for helpful suggestions, and Nandagopal Ramachandran for helpful discussions. The second author thanks Chia-Fu Yu for helpful discussions. 
\end{acknowledgement}

\section{Reinterpretation of the $\Fp$-roots}\label{sec:Fp-roots}

As mentioned before, we are going to reinterpret the
$\Fp$-roots of $\wtH_\mathcal{O}(x)$ in terms of quaternion
orders.  For this purpose,  we first describe more concretely the reduction of singular moduli with complex multiplication by $\calO$. Assume that the prime $p$  is non-split in $K$. For the moment, we make no assumption
on the discriminant of the order $\calO\subset K$. 

Let $\ells(\calO)$ be the set of isomorphism classes of elliptic
curves over $\Qbar$ with complex multiplication by $\calO$. It is
canonically identified with the singular $\grj$-invariants with
complex multiplication by $\calO$ (i.e.~the roots of
$H_\calO(x)\in \ZZ[x]$). The Picard group $\Pic(\calO)$ acts regularly
on $\ells(\calO)$ via $\gra$-transformation \cite[\S7]{Shimura-CM} and
\cite[\S1]{milne2007fundamental}:
\begin{equation}
  \label{eq:1}
  \Pic(\calO)\times \ells(\calO)\to \ells(\calO),\qquad ([\gra],
  E)\mapsto E^\gra. 
\end{equation}
More concretely, if we pick $\gra$ to be an integral ideal of $\calO$ 
and write $E[\gra]$ for the finite group scheme $\cap_{a\in
  \gra}E[a]$,  
then $ E^\gra=E/E[\gra]$ by \cite[Corollary~A.4]{waterhouse_thesis}. Here $E[a]=\ker(E\xrightarrow{a}E)$. See \cite[Proposition~1.26]{milne2007fundamental}
  and \cite[Appendix]{waterhouse_thesis} for the functorial characterization
  of $E^\gra$. Alternatively, since $\gra$ is an invertible $\calO$-ideal, 
  $E^\gra$ can also be identified canonically with the Serre tensor construction
  $\gra^{-1}\otimes_\calO E$ (see \cite[\S1]{Amir-Khosravi} and
  \cite[\S1.7.4]{Chai-Conrad-Oort}). Fix a member $E_0\in
\ells(\calO)$. The regular action in (\ref{eq:1}) gives rise to a
$\Pic(\calO)$-equivariant bijection $\xi: \ells(\calO)\to
\Pic(\calO)$ that sends $E_0$ to the identity element $[\calO]\in
\Pic(\calO)$.

Similarly, let $\els$ be the set of isomorphism classes of
supersingular elliptic curves over $\Fpbar$, which is canonically
identified with the set of supersingular $\grj$-invariants in
$\Fpp$. From \cite[Theorem~V.3.1]{Silverman-AEC-I}, an elliptic curve
$E/\Fpbar$ is supersingular if and only if its endomorphism algebra
$\End^0(E):=\End(E)\otimes \QQ$ is a quaternion $\QQ$-algebra.  Assume
that this is the case.  Then $\End^0(E)$ coincides with the unique
quaternion $\QQ$-algebra $B_{p, \infty}$ ramified precisely at $p$ and
infinity, and $\End(E)$ is a maximal order in $\End^0(E)$ by
\cite[Theorem~4.2]{waterhouse_thesis}. For simplicity, put
$B:=B_{p, \infty}$ and let $\typ(B)$ be the \emph{type set} of $B$,
that is, the set of isomorphism (i.e.~$B^\times$-conjugacy) classes of maximal orders in
$B$. We obtain the following canonical map, which is known to be surjective \cite[Corollary~42.2.21]{voight}: 
\begin{equation}
  \label{eq:6}
 \rho: \els\twoheadrightarrow \typ(B), \qquad E\mapsto [\End(E)].   
\end{equation}
Let $\calR$ be a maximal order in $B$, and $\Cl(\calR)$ be its
 left ideal class set, that is, the set of isomorphism (i.e.~right $B^\times$-equivalent) classes of
fractional left ideals of $\calR$ in $B$.  Given a fractional left ideal $I$ of
$\calR$, we write $\calR_r(I)$ for the right order of $I$, which is  defined as follows:
\[\calR_r(I):=\{x\in B\mid Ix\subseteq I\}.\]
Sending a fractional left $\calR$-ideal to its right order  induces a surjective map
\begin{equation}
  \label{eq:5}
  \Upsilon: \Cl(\calR)\twoheadrightarrow \typ(B), \qquad [I]\mapsto [\calR_r(I)]. 
\end{equation}

The Deuring correspondence
\cite[Corollary~42.3.7]{voight} establishes a bijection between
$\Cl(\calR)$ and $\els$. One direction of this correspondence goes as
follows. From the surjectivity of $\rho$, we may always fix 
$E_\calR\in \els$ such that $\End(E_\calR)=\calR$. Then the member of
$\els$ corresponding to a left ideal class $[I]\in \Cl(\calR)$ is
the  $I$-transform $E_\calR^I$ of $E_\calR$.  If $I$ is chosen to be
an integral left ideal of $\calR$, then $E_\calR^I$ 
can be identified with the 
quotient $E_\calR/E_\calR[I]$ by
\cite[Corollary~A.4]{waterhouse_thesis} again. From
\cite[Corollary~42.3.7]{voight}, we have \begin{equation}\label{eq:order}
    \End(E_\calR^I)\simeq
\calR_r(I). 
\end{equation}

Let $\grP$ be a place of $\Qbar$ lying above $p$, and   $r_\grP:
\ells(\calO)\to \els$ be the reduction map modulo $\grP$. For each
$E\in \ells(\calO)$, we write $\wtE$ for the reduction of $E$ modulo
$\grP$. From
\cite[\S9.2]{lang-ell-func}, reducing $E_0$ modulo $\grP$ gives rise to an embedding
$\iota:\calO\hookrightarrow \calR_0:=\End(\wtE_0)$.   
By an abuse of notation, we still
write $\iota$ for both of the following two induced maps:
\begin{equation}
  \label{eq:4}
K\hookrightarrow
B\quad \text{and}\quad  \Pic(\calO)\xrightarrow{[\gra]\mapsto
  [\calR_0\iota(\gra)]} \Cl(\calR_0).
\end{equation}
For simplicity,  we  identify $K$ with its image in $B$ via $\iota$ and
 write $\calR_0 \gra$ for $\calR_0 \iota(\gra)$. 


Now we are ready to give a concrete description of  $r_\grP:\ells(\calO)\to \els$.

\begin{proposition}\label{prop:diagram}
The reduction map $r_\grP$ fits into a commutative diagram as
follows:
\[
  \begin{tikzcd}
    \ells(\calO)\ar[d, leftrightarrow,
"\simeq", "\xi"'] \ar[r, "r_\grP"] & \els \ar[d, leftrightarrow,
 "\simeq", "\delta"'] \ar[rd, twoheadrightarrow, "\rho"]& \\
    \Pic(\calO)\ar[r,"\iota"] & \Cl(\calR_0)\ar[r, twoheadrightarrow, "\Upsilon"]  & \typ(B).
  \end{tikzcd}
\]
Here $\xi$ is the $\Pic(\calO)$-equivariant bijection that sends
the fixed member $E_0\in \ells(\calO)$ to  $[\calO]\in
\Pic(\calO)$, and $\delta$ is the Deuring correspondence obtained by
taking $E_{\calR_0}=\wtE_0$. 
\end{proposition}


\begin{proof}
According to  \cite[Proposition~15, \S11]{Shimura-CM}, $\gra$-transforms are
preserved under good reductions\footnote{A priori, the statement of 
  \cite[Proposition~15, \S11]{Shimura-CM} requires that 
  $\calO=\calO_K$, the maximal order of $K$.  Nevertheless, the result here
  holds for general $\calO$ here
  since $\gra$ is an
  invertible $\calO$-ideal by our assumption.}. This implies that for
every $[\gra]\in \Pic(\calO)$, we have 
\[\widetilde{E_0^\gra}=(\wtE_0)^\gra=(\wtE_0)^{\calR_0\gra}, \]
so the  left square commutes. The right triangle commutes because of \eqref{eq:order}.
\end{proof}


\begin{corollary}\label{cor:End-Ea}
  For any $[\gra]\in \Pic(\calO)$, we have
  $\End(\widetilde{E_0^\gra})\simeq \gra^{-1}\calR_0\gra$. 
\end{corollary}
\begin{proof}
  This follows directly  from Proposition~\ref{prop:diagram} since the
  right order of 
  $\calR_0\gra$ is precisely $\gra^{-1}\calR_0\gra$. 
\end{proof}

\begin{remark}\label{rem:conductor}
Let $\calO_K$ be the ring of integers of $K$, and $f$ be the conductor
of $\calO$ so that $\calO=\ZZ+f\calO_K$. Write $f=p^mf'$ with $p\nmid
f'$, and put $\calO':=\ZZ+f'\calO_K$. According to
\cite[Lemma~3.1]{onuki2020oriented}, 
$\iota(K)\cap \calR_0= \iota(\calO')$. 
For any invertible fractional ideal $\gra$ of $\calO$, we have
$\calR_0 \gra=(\calR_0
\calO')\gra=\calR_0(\calO'\gra)$. It follows that
the map $\iota: \Pic(\calO)\to \Cl(\calR_0)$ factors through the
following 
canonical homomorphism 
\[\varpi: \Pic(\calO)\to \Pic(\calO'),\qquad [\gra]\mapsto [\calO'\gra].\]
From this, one easily deduces that
$\wtH_\calO(x)=(\wtH_{\calO'}(x))^{|\ker(\varpi)|}$.
\end{remark}

Now assume that $\calO$ is maximal at $p$ (i.e.~$p\nmid
f$). From Remark~\ref{rem:conductor},  $\iota: \calO\to \calR_0$ is an \emph{optimal embedding} of
$\calO$ into $\calR_0$, that is, $\iota(K)\cap \calR_0= \iota(\calO)$. 
Given  an arbitrary maximal order $\calR$ of $B$, we write
$\Emb(\calO, \calR)$ for the set of optimal embeddings of $\calO$ into
$\calR$. The unit group $\calR^\times$ acts on $\Emb(\calO, \calR)$ by conjugation, and there are only finitely many
orbits. Put
$m(\calO, \calR, \calR^\times):=|\calR^\times\bsh \Emb(\calO,
  \calR)|$, the number of $\calR^\times$-conjugacy clasess of optimal embeddings from $\calO$ into $\calR$. We recall below a precise formula by   Elkies, Ono and Yang
   for the cardinality of each
  fiber of the reduction map $r_\grP:\ells(\calO)\to \els$. 

  \begin{lemma}[{\cite[Lemma~3.3]{Elkies-Ono-Yang}}]\label{lem:elkies}
   Suppose that $\calO$ is maximal at $p$. Then for any member $E\in
   \els$, we have 
\[|r_\grP^{-1}(E)|=\varepsilon\cdot m(\calO, \calR, \calR^\times), \]
where $\calR=\End(E)$, and $\varepsilon=1/2$ or $1$ according as $p$
is inert or ramified in $K$. 
\end{lemma}
A priori, \cite[Lemma~3.3]{Elkies-Ono-Yang} is only stated for the maximal
order $\calO_K$. Nevertheless, the same proof there applies more
generally to quadratic orders maximal at $p$. Alternatively,
using Proposition~\ref{prop:diagram} and the Deuring lifting
theorem\footnote{Here the Deuring lifting theorem guarantees that the
  optimal embedding $\iota: \calO\to \calR_0$ is ``non-special'', that is,
  every optimal embedding $\varphi: \calO\to \calR$ 
  is realizable as $\End(E)\to \End(\wtE)$ for some
  $E\in \ells(\calO)$.} \cite[Theorem~14, \S13.5]{lang-ell-func}
   \cite[Proposition~2.7]{GrossZagier}, 
one easily sees that Lemma~\ref{lem:elkies} is equivalent to 
 to the following
purely arithmetic result, whose independent proof
will be left for the interested reader. 
\begin{lemma}
Keep $\calO$ and $\varepsilon$  as in Lemma~\ref{lem:elkies}. 
Let $\calR$ be a maximal order in $B$, and $\varphi:
   \calO\hookrightarrow \calR$ be an 
 optimal embedding. Denote the
   induced map $\Pic(\calO)\to \Cl(\calR)$  by
   $\varphi$ as well. Then for each $[I]\in \Cl(\calR)$, we have
   \[|\varphi^{-1}([I])|=\varepsilon\cdot m(\calO, \calR_r(I),
     \calR_r(I)^\times). \]
\end{lemma}






We immediately obtain the following corollaries from
Lemma~\ref{lem:elkies}. 

\begin{corollary}\label{optimal}
Suppose that $\calO$ is maximal at $p$. 
  The $\grj$-invariant of a supersingluar elliptic curve $E/\Fpbar$ is
  a root of $\wtH_\mathcal{O}(x)$ if and only if $\mathcal{O}$ can be
  optimally embedded into $\End(E)$.
\end{corollary}

This matches well with Corollary~\ref{cor:End-Ea}. Indeed, a classical
result of Chevalley, Hasse and Noether 
\cite[\S4]{Ibukiyama-max-order-1982} says that  any maximal order of $B$ that
contains a copy of $\calO$ optimally is isomorphic to $\gra^{-1}\calR_0\gra$ for
some $[\gra]\in \Pic(\calO)$.

\begin{corollary}\label{cor:no-repeat-roots}
If $p>|\disc(\calO)|$, then the reduction map $r_\grP:\ells(\calO)\to
\els$ is injective. In particular, $\wtH_\calO(x)$ has no repeated 
roots.  
\end{corollary}
We give  a simple proof that is independent of the result of Gross and
Zagier \cite{GrossZagier}. 
\begin{proof}
  Since $p$ does not split in $K$ and is strictly greater than
  $|\disc(\calO)|$, it is necessarily inert in $K$.   From Lemma~\ref{lem:elkies}, it suffices to show that $|\Emb(\calO,
  \calR)|\leq 2$ for any maximal order $\calR$ in $B$.  Since
  $p>|\disc(\calO)|$, Kaneko's inequality  \cite[Theorem~2']{Kaneko}
  forces any  two optimal embeddings   $\varphi, \varphi': \calO\to
  \calR$ to have
  the same image. On the other hand, $\varphi$ and $\varphi'$ share the same
  image if and only if $\varphi'=\varphi$ or $\bar\varphi$, the
  complex conjugate of $\varphi$. The desired inequality $|\Emb(\calO,
  \calR)|\leq 2$ follows immediately. 
\end{proof}

\begin{remark}
  In another direction,  Elkies, Ono and Yang
  \cite[Theorem~1.4]{Elkies-Ono-Yang} showed that there exists a bound
  $N_p$ such that the reduction map $r_\grP:\ells(\calO_K)\to
\els$ is surjective whenever $|\disc(\calO_K)|>N_p$. This bound is
first effectivized  by Kane \cite{Kane-2009} conditionally  upon
the generalized Riemann hypothesis. Liu et al.~further improved this
bound in 
~\cite[Corollary~1.3]{Liu-Masri-Young}. 
\end{remark}

Let us return to the task of interpreting $\Fp$-roots of
$\wtH_\calO(x)\in \Fp[x]$ in terms of maximal orders in $B$. For the rest of this section, we keep the
additional assumption 
that $p>|\disc(\calO)|$. We recall from \cite[Proposition 2.4]{DG16} a classical result on supersingular elliptic curves in characteristic $p$.

\begin{lemma}\label{Fpsuperj}
Let $p>3$ and let $E$ be a supersingular elliptic curve over $\overline{\mathbb{F}}_p$. Then $\grj(E)\in \Fp$ if and only if there exists $\psi\in \End(E)$ such that $\psi^2 = -p$.
\end{lemma}


Recall that $\calH_p$ denotes the set of   $\Fp$-roots of
$\wtH_\calO(x)$, which can be identified canonically     with a subset of $\els$.

\begin{lemma}\label{lem:bijTp}
 The map $\rho: \els\to
\typ(B)$ in (\ref{eq:6}) induces a bijection between $\calH_p$ and the
following subset $\calT_p\subseteq \typ(B)$:
\begin{equation}
  \label{eq:10}
  \calT_p:=\{[\calR]\in \typ(B)\mid \Emb(\calO, \calR)\neq \emptyset,
  \text{ and }\exists \alpha\in \calR \text{ such that } \alpha^2=-p\}.
\end{equation}
\end{lemma}
\begin{proof}
  Combining
Corollary~\ref{optimal} and Lemma~\ref{Fpsuperj}, we see that
$\rho(\calH_p)=\calT_p$. Now it follows from
\cite[Lemma 42.4.1]{voight} that $\rho: \calH_p\to \calT_p$ is
injective, and hence bijective.  
\end{proof}

We give another characterization of $\calT_p$ by presenting
the quaternion algebra $B=B_{p, \infty}$ more concretely. Let $d\in \mathbb{N}$ be the unique square-free positive integer such that $K=\QQ(\sqrt{-d})$. The assumption that $p$ is inert in $K$ amounts to the equality $\legendre{-d}{p}=-1$.
Let $\left( \frac{-d,-p}{\QQ} \right)$ be the quaternion $\QQ$-algebra
with standard basis $\{1, i,j,k\}$ such that
\begin{equation}
  \label{eq:2}
  i^2 = -d,\quad  j^2 = -p\quad \text{and}\quad k = ij = -ji.
\end{equation}
 We  identify $K=\QQ(\sqrt{-d})$ with $\QQ(i)$,
 and $\calO$ with the corresponding order in $\QQ(i)$. 
 Put
 $\Lambda:=\calO+j \calO$, which is  an order (of full rank) in the
 above quaternion algebra. 
 Consider the following finite set of maximal orders: 
 \begin{equation}
   \label{eq:8}
S^\opt := \left\{\scrR \subset \left( \frac{-d,-p}{\QQ} \right)
    \middle\vert \,
  \scrR\text{ is a maximal order containing }\Lambda \text{ and } \scrR\cap\QQ(i) = \mathcal{O}\right\}.   
 \end{equation}
Here the superscript ``opt" stands for ``$\mathcal{O}$-optimal".


\begin{proposition}\label{prop:Sopt}
Let $\calR$ be a maximal order in $B$. We have $[\calR]\in \calT_p$ if
and only if 
$\calR\simeq \scrR$ for some $\scrR\in S^\opt$. In particular,
$\calH_p\neq \emptyset$ if and only if $ \left(
  \frac{-d,-p}{\QQ} \right)\simeq B$ and $S^\opt\neq \emptyset$.
\end{proposition}

\begin{proof}
  Clearly, if $\calR\simeq \scrR$ for some $\scrR\in S^\opt$, then
  $[\calR]\in \calT_p$. Conversely, suppose that $[\calR]\in \calT_p$,
  that is, $\calR$ contains a copy of $\calO$ optimally, and there
  exists $\alpha\in \calR$ with $\alpha^2=-p$. Then
  $\calR\alpha$ is the unique two sided prime ideal of $\calR$ lying
  above $p$.  From  \cite[Exercise~I.4.6]{vigneras}, 
 $\calR$ is
  normalized by $\alpha$, which implies that 
  $\calO_\alpha:=\alpha\mathcal{O}\alpha^{-1}$ is still a quadratic order
  optimally embedded in $\mathcal{R}$.  If $\calO_\alpha\neq \calO$, then
  $|\disc(\calO)|\geq p$ by Kaneko's inequality \cite[Theorem
  2']{Kaneko}, contradicting to our assumption that
  $|\disc(\calO)|<p$. Thus $\calO_\alpha=\calO$, and conjugation by $\alpha$
  induces an automorphism $\sigma\in \Aut(\calO)$.  If $\sigma$ is the
  identity, then $\alpha$ lies in the centralizer of
  $\calO$ in $B$, which is just $K$. This  contradicts to the
  assumption 
  $|\disc(\calO)|<p$ again.  It follows that $\sigma$ is the unique
  nontrivial automorphism of $\calO$, i.e.~the complex conjugation. 
 We conclude that 
  $\Lambda_\calR:=\calO+\alpha \calO\subset \calR$ is isomorphic to $\Lambda$, and
 $B=\Lambda_\calR\otimes_\ZZ \QQ\simeq \Lambda\otimes_\ZZ\QQ= \left(
  \frac{-d,-p}{\QQ} \right)$. Consequently, $\calR$ is 
  isomorphic to some member of $S^\opt$. The last statement follows
  from the bijection $\calH_p\simeq \calT_p$  in  Lemma~\ref{lem:bijTp}.
\end{proof}

\begin{lemma}\label{lem:isom-B-p-infty}
The isomorphism $\left( \frac{-d,-p}{\QQ} \right)\simeq B$ holds if and only if $\legendre{-p}{\ell}=1$ for every odd prime factor $\ell$ of $d$.
\end{lemma}
\begin{proof}
For the moment, let $\ell$ be either a prime number or $\infty$. Write $(-d, -p)_\ell$ for the Hilbert symbol of $-d$ and $-p$ relative to $\QQ_\ell$ (where $\QQ_\infty=\RR$). From \cite[Corollaire~II.1.2]{vigneras}, $\left( \frac{-d,-p}{\QQ} \right)$ is split at $\ell$ if and only if $(-d, -p)_\ell=1$.
Clearly, $(-d, -p)_\infty=-1$. 

Now assume that $\ell$ is an odd prime. By our assumption,  $p$ is an odd prime satisfying $\legendre{-d}{p}=-1$.  From \cite[Theorem~1, \S III.1]{Serre-arithmetic}, we easily compute that 
\[(-d, -p)_\ell=\begin{cases}
1&\text{if } \ell\nmid (dp);\\
-1& \text{if } \ell=p;\\
\legendre{-p}{\ell} &\text{if } \ell|d. 
\end{cases}\]
  Therefore, if $\left( \frac{-d,-p}{\QQ} \right)\simeq B$, then necessarily $\legendre{-p}{\ell}=1$ for every odd prime factor $\ell$ of $d$.

Conversely, if $\legendre{-p}{\ell}=1$ for every odd prime factor $\ell$ of $d$,   then $(-d, -p)_2=1$ by the product formula \cite[Theorem~2, \S III.2]{Serre-arithmetic}. Hence
this condition is also sufficient for the  isomorphism $\left( \frac{-d,-p}{\QQ} \right)\simeq B$. 
\end{proof}

\section{The $\Pic(\mathcal{O})[2]$-action on $\calH_p$ and the
  nonemptiness criterion}\label{sec:free_transitive}

Throughout this section, we assume that $p$ is inert in $K=\QQ(\sqrt{-d})$
and strictly greater than $|\disc(\calO)|$. 
Assume further  that the quaternion $\QQ$-algebra $\left(
  \frac{-d,-p}{\QQ} \right)$ is ramified precisely at $p$ and
infinity,  for otherwise $\calH_p=\emptyset$. Denote $\left(
  \frac{-d,-p}{\QQ} \right)$ simply by $B$
henceforth and let $\{1, i, j, k\}$ be the standard basis of $B$ as in
(\ref{eq:2}). We identify $K$ with the subfield $\QQ(i)$ of $B$. Then
 conjugation by $j$ stabilizes $K$ and sends each $x\in K$ to its
complex conjugate $\bar{x}$. Let $\Lambda=\calO+j\calO$, and $S^\opt$ be
the set of maximal orders in (\ref{eq:8}).



First, we assume that $\calH_p\neq \emptyset$ and  exhibit a regular action of $\Pic(\calO)[2]$ on
$\calH_p$. Since the reduction map
$r_\grP: \ells(\calO)\to \els$ is injective by
Corollary~\ref{cor:no-repeat-roots}, the regular action  of
$\Pic(\calO)$ on
$\ells(\calO)$ induces a regular action of $\Pic(\calO)$ on the image
$r_\grP(\ells(\calO))$ (or equivalently, on the full set of roots of
$\wtH_\calO(x)$).  We show that this action restricts to a regular 
$\Pic(\calO)[2]$-action on $\calH_p$.


\begin{proposition}\label{prop:picO2}
  Let $E_0\in \ells(\calO)$ be a member satisfying $\grj(\wtE_0)\in
  \Fp$. Given $[\gra]\in \Pic(\calO)$, we 
 have $\grj(\widetilde{E_0^\gra})\in
\Fp$   if and only if $[\gra]$ is a 2-torsion. In particular,
$\Pic(\calO)[2]$ acts regularly on $\calH_p$. 
\end{proposition}

\begin{proof}
  Put $\calR_0: =\End(\wtE_0)$ and $\calR:=\gra^{-1}\calR_0\gra$ 
  so that
  $\End(\widetilde{E_0^\gra})\simeq \calR$ by
  Corollary~\ref{cor:End-Ea}. From Lemma~\ref{Fpsuperj}, it is enough
  to show that there exists $\alpha\in \calR$ with $\alpha^2=-p$ if and only if $[\gra]\in
  \Pic(\calO)[2]$. By Proposition~\ref{prop:Sopt}, we may assume that 
  $\calR_0\in S^\opt$, that is, $\calR_0$ is a
  maximal order in $B$ satisfying  $\calR_0\supseteq
  \calO+j \calO$ and $\calR_0\cap K=\calO$. 
  Then 
  \begin{equation}
    \label{eq:3}
\calR\cap K=\gra^{-1}(\calR_0\cap K)\gra =
\mathcal{O},\quad\text{and}\quad  \calR\supseteq \gra^{-1} j \gra=\gra^{-1}\bar\gra j.   
  \end{equation}

First, suppose that $[\gra]\in \Pic(\calO)[2]$. Then
$\mathfrak{a}^{-1}\bar{\mathfrak{a}} = \calO a$ for some $a\in
K^\times$. Moreover, $N_{K/\QQ}(\calO a) =
N_{K/\QQ}(\gra^{-1}\bar\gra)=\ZZ$, so $N_{K/\QQ}(a)=1$.  Therefore $\alpha: =aj\in
\mathcal{R}$ satisfies that $\alpha^2 = a\bar{a}j^2 = -p$.

Conversely, suppose that $\alpha \in \calR$ is an element satisfying
$\alpha^2=-p$. From the proof of Proposition~\ref{prop:Sopt}, we must have 
$\alpha x= \bar{x}\alpha$ for every $x\in \calO$. Thus $j^{-1}\alpha$ centralizes $\calO$, so  
there exists $a\in K^\times$ such that $\alpha=ja$. Moreover,
$N_{K/\QQ}(a)=1$ since 
$\alpha^2=j^2\bar{a}a$. Now we have 
\begin{equation}
  \label{eq:7}
\calR\supset \gra^{-1}\bar\gra j\cdot \alpha=\gra^{-1}\bar\gra j\cdot ja=-pa\gra^{-1}\bar\gra. 
\end{equation}
We claim that  $\calR\supset  a\gra^{-1}\bar\gra$. It suffices to
show that $\calR_\ell\supset a\gra_\ell^{-1}\bar\gra_\ell$ for every
prime $\ell\in \mathbb{N}$, where the subscript $_\ell$ indicates $\ell$-adic
completion at $\ell$. If $\ell\neq p$, then $(-p)\in
\calR_\ell^\times$, so the containment follows directly from
(\ref{eq:7}). If $\ell=p$, then $\calR_p$ coincides with the unique
maximal order of the division quaternion $\QQ_p$-algebra $B_p$. More concretely, $\calR_p=\{z\in B_p\mid \nrd(z)\in
\ZZ_p\}$, where $\nrd(z)$ denotes the reduced norm of $z\in B_p$.   On the
other hand,  for any $x_p\in \gra_p^{-1}$ and $y_p\in \gra_p$, we have
$x_py_p\in \calO_p$, and hence 
$\nrd(a x_p\bar{y}_p)=\nrd(x_p)\nrd(\bar{y}_p)=\nrd(x_py_p)\in
\ZZ_p$. Since $\gra_p^{-1}\bar\gra_p$ is generated by elements of the
form $x_p\bar{y}_p$, it follows  that $\calR_p\supset
a\gra_p^{-1}\bar\gra_p$. The claim is verified. Now $
a\gra^{-1}\bar\gra\subseteq \calR\cap
K=\calO$, which implies that $
a\bar\gra\subseteq \gra$. Comparing  discriminants on both sides, we get
$\disc(a\bar\gra)=N_{K/\QQ}(a)^2\disc(\bar{\gra})=\disc(\gra)$. Therefore,
$a\bar\gra=\gra$, so $[\gra]\in \Pic(\calO)[2]$. 
\end{proof}

Now we drop the assumption that $\calH_p\neq \emptyset $ and derive a
non-emptiness criterion for $\calH_p$. From
Proposition~\ref{prop:Sopt}, $\calH_p\neq \emptyset$ if and only if
$S^\opt\neq \emptyset$ (as we have already assumed that $\left(
  \frac{-d,-p}{\QQ} \right)\simeq B_{p, \infty}$). For each prime
$\ell\in \mathbb{N}$, let us put
 \[S_\ell^\opt := \{\calR_\ell \subseteq B_\ell\mid\calR_\ell \text{
     is a maximal order containing } \Lambda_\ell \text{ and }
   \mathcal{R}_\ell\cap K_\ell = \mathcal{O}_\ell\}.\]
The local-global correspondence of lattices \cite[Proposition~4.21]{curtis-reiner:1}  establishes a
bijection between 
$S^{\text{opt}}$ and $\prod_\ell
  S_\ell^{\text{opt}}$, where the product runs over all prime $\ell$. 
 Since the reduced discriminant of $B$ is $p$ and the reduced
 discriminant of $\Lambda$ is $p\disc(\calO)$ by \cite[Lemmas~2.7 and
2.9]{LXY}, $\Lambda$ is maximal at every prime $\ell$ coprime to $\disc(\calO)$. Moreover, for each such $\ell$, the maximal order $\Lambda_\ell$
automatically satisfies the condition
$\Lambda_\ell\cap K_\ell = \mathcal{O}_\ell$ by its definition $\Lambda_\ell=\calO_\ell+j\calO_\ell$. 
Hence
for $\ell \nmid \disc(\calO)$, the set $S_\ell^\opt$ has a
single element $\Lambda_\ell$, and the bijection above  simplifies as 
\begin{equation}
  \label{eq:13}
  S^\opt\longleftrightarrow \prod_{\ell|\disc(\calO)}S_\ell^\opt.
\end{equation}

 

\begin{lemma}\label{lem:Sopt-nonempty}
  Let $\ell$ be a prime factor of $\disc(\calO)$. Then
  $S_\ell^\opt \neq \emptyset$ if and only if
  $-p\in N_{K/\QQ}(\calO_\ell^\times)$.  Moreover, if $S_\ell^\opt\neq \emptyset$,
  then there is a regular action of $H^1(K/\QQ,
  \calO_\ell^\times)$ on $S_\ell^\opt$, so any fixed member of
  $S_\ell^\opt$ gives rise to a bijection  $S_\ell^\opt\simeq H^1(K/\QQ,
  \calO_\ell^\times)$. 
\end{lemma}

The Galois cohomological description of $S_\ell^\opt$ is  nice to
know but not used elsewhere
in this paper. 

\begin{proof}
  By our assumption, $\disc(\calO)$ is coprime to $p$, so $B$ splits
  at the prime $\ell$. This allows us to  identify $B_\ell$ with the matrix 
  algebra $M_2(\QQ_\ell)$.  Let $V_\ell=\QQ_\ell^2$ be the unique
  simple $B_\ell$-module. 
  Every maximal order $\calR_\ell$ in
  $B_\ell$ is of the form $\End_{\ZZ_\ell}(L_\ell)$ for some
  $\ZZ_\ell$-lattice $L_\ell\subseteq V_\ell$, and $L_\ell$
  is uniquely determined by $\mathcal{R}_\ell$ up to $\QQ_\ell^\times$-homothety. In
  other words, $\End_{\ZZ_\ell}(L_\ell) = \End_{\ZZ_\ell}(L'_\ell)$ if
  and only if $L_\ell =c L'_\ell$ for some
  $c \in \QQ_\ell^\times$.  If $\calR_\ell \in S_\ell^\opt$,
  then the inclusion
  $\Lambda_\ell \subseteq \mathcal{R}_\ell$ puts a
  $\Lambda_\ell$-module structure on   $L_\ell$. Moreover, the $\Lambda_\ell$-lattice $L_\ell$  is
  \emph{$\calO_\ell$-optimal} in the sense that
  $\End_{\ZZ_\ell}(L_\ell)\cap K_\ell =
  \mathcal{O}_\ell$.  Conversely, if $M_\ell$ is an  $\calO_\ell$-optimal
  $\Lambda_\ell$-lattice  in $V_\ell$, then
  $\End_{\ZZ_\ell}(M_\ell)$ is a member of 
  $S_\ell^\opt$. We have established the following  canonical bijection
  \begin{equation}
    \label{eq:11}
    \mathcal{S}_\ell^\text{opt} \longleftrightarrow
    \mathcal{M}:=\{\mathcal{O}_\ell\text{-optimal } \Lambda_\ell\text{-lattices } L_\ell\subset V_\ell \}/\QQ_\ell^\times.
  \end{equation}

Recall that $\Lambda_\ell=\calO_\ell+j\calO_\ell$, where $j^2 = -p$ and $jx = \bar{x}j$ for
  any $x\in \mathcal{O}_\ell$.  If there exists $a\in \calO_\ell^\times$ satisfying
$a\bar{a}=-p$, then we can put a $\Lambda_\ell$-module structure on
$\calO_\ell$ as follows: 
\[ (x+jy)\cdot z= xz+\bar{y}\bar{z}a,  \qquad\forall x, y, z\in \calO_\ell. \]
Since $B_\ell=\Lambda_\ell\otimes_{\ZZ_\ell}\QQ_\ell$, this also puts
a $B_\ell$-module structure on $K_\ell=\calO_\ell\otimes_{\ZZ_\ell}\QQ_\ell$. Consequently, it identifies $K_\ell$
with the unique simple $B_\ell$-module $V_\ell$, and in turn
identifies $\calO_\ell$  with  a $\Lambda_\ell$-lattice $L_\ell$ in
$V_\ell$. Necessarily, $L_\ell$ is $\calO_\ell$-optimal since
$\End_{\ZZ_\ell}(L_\ell)\cap
K_\ell=\End_{\calO_\ell}(L_\ell)=\End_{\calO_\ell}(\calO_\ell)=\calO_\ell$.
We have shown that if $-p\in N_{K/\QQ}(\calO_\ell^\times)$, then
$S_\ell^\opt\neq \emptyset$.
  
Conversely,  suppose that $S_\ell^\opt\neq \emptyset$ and let $M_\ell$ be an $\calO_\ell$-optimal $\Lambda_\ell$-lattice
  in $V_\ell$. The inclusion $\calO_\ell\subset \Lambda_\ell$
  equips $M_\ell$ with an $\calO_\ell$-module structure satisfying
  $\End_{\calO_\ell}(M_\ell)=\calO_\ell$. Being a quadratic
  $\ZZ_\ell$-order, $\calO_\ell$ is both Gorenstein and semi-local. It
  follows from \cite[Characterization B 4.2]{gorenstein} that $M_\ell$
  is a free $\mathcal{O}_\ell$-module of rank one. Pick  a basis $e$
  so that $M_\ell = \mathcal{O}_\ell e$.  Since $M_\ell$ is at the
  same time a module over $\Lambda_\ell$, we
  have 
 $je = ae$ for some
  $a\in \mathcal{O}_\ell$. Necessarily, $\bar{a}a=-p$ because 
 \[-pe = j^2e = j(je) = j(ae) = \bar{a}je = \bar{a}ae. \]
This also implies that $a\in \mathcal{O}_\ell^\times$ since $\ell\neq p$.   Therefore, $S_\ell^\text{opt}
\neq \emptyset$ if and only if  $-p\in N_{K/\QQ}(\calO_\ell^\times)$.

 Had we picked a different basis $e'$ for $M_\ell$, then $e'=ue$
  for some $u\in
  \mathcal{O}^\times_\ell$. It follows that 
  \[j e'=j(ue)=\bar{u}je=\bar{u}ae=u^{-1}\bar{u}ae'.\] 
Correspondingly, $a$ is changed to $u^{-1}\bar{u}a$. Therefore,
  we have defined the following map:
  \begin{equation}
    \label{eq:12}
\Phi: \mathcal{M} \rightarrow \{a\in \calO_\ell^\times\mid a\bar{a} =
-p\}/{\sim},
  \end{equation}
where $a\sim a'$ if and only if there exists some
  $u\in \mathcal{O}_\ell^\times$ such that $a'=a(\bar{u}/u)$. We have
  already seen that $\Phi$ is surjective. Suppose that 
  $\Phi([M_1])=\Phi([M_2])$ for $[M_r]\in \calM$ with $r=1, 2$.  By the
  above discussion, we can choose suitable $\calO_\ell$-base $e_r$
  for $M_r$ such that they give rise to the same $a\in
  \calO_\ell^\times$. Then the $\calO_\ell$-linear map sending $e_1$ to $e_2$
  defines a $\Lambda_\ell$-isomorphism between $M_1$ and
  $M_2$. Since $\Aut_{B_\ell}(V_\ell)=\QQ_\ell^\times$,  it follows
  that  $M_1$ and
  $M_2$ are $\QQ_\ell^\times$-homothetic, so   $\Phi$ is injective as
  well.

  Lastly, if the right hand side of (\ref{eq:12}) is nonempty, then it
  admits a regular action by   $H^1(K/\QQ,\calO_\ell^\times)=\{b\in
  \calO_\ell^\times \mid \bar{b}b=1\}/{\sim}$ via multiplication.  The
  second part of  the lemma follows by combining the bijections
  (\ref{eq:11}) and (\ref{eq:12}) with the above action.  
\end{proof}












\begin{lemma}\label{lem:-p=norm}
Let $\ell$ be a prime factor of $\disc(\calO)$.  Then $-p\in N_{K/\QQ}(\calO_\ell^\times)$ if and only if either condition (i) or (ii) below holds for $\ell$ depending on  its parity:
\begin{enumerate}
    \item[(i)] $\ell \neq 2$ and $\legendre{-p}{\ell} = 1$;
    \item[(ii)] $\ell = 2$ and one of the following conditions holds:
        \subitem (a) $p \equiv 7 \amod{8}$;
        \subitem (b) $-p + \frac{\disc(\calO)}{4} \equiv 0,\,1 $ or $4 \amod{8}$;
        \subitem (c) $-p + \disc(\calO)\equiv 1 \amod{8}$.
\end{enumerate}
\end{lemma}

\begin{proof} 
 For simplicity, put
  $D:=\disc(\calO)$ and $\delta=\frac{1}{2}\sqrt{D}$. We claim that
  $\calO_\ell=\ZZ_\ell+\ZZ_\ell\delta$. It is well known that
  $\calO=\ZZ+\ZZ(D+\sqrt{D})/2$. The claim is obviously
  true if $4|D$. If $4\nmid D$, then $\ell\neq 2$, so the claim
  is true in this case as well. Given an element
  $a+b\delta\in \calO_\ell$ with $a, b\in \ZZ_\ell$, we have
  $N_{K/\QQ}(a+b\delta)=a^2-b^2D/4$. Therefore, $-p \in N_{K/\QQ}(\mathcal{O}_\ell^\times)$ if and only if the equation 
    \begin{equation}\label{eq:norm_eq}
         x^2 - y^2\frac{D}{4} = -p
    \end{equation}
   has a solution in $\ZZ_\ell^2$.

First, suppose that $\ell$ is odd. Then equation (\ref{eq:norm_eq}) is
solvable in $\ZZ_\ell^2$ if  and only if $\legendre{-p}{\ell} =
1$. Indeed, suppose  $\legendre{-p}{\ell} = 1$ so that $-p$ is a square in
$\mathbb{F}_\ell$. By Hensel's lemma \cite[Lemma~12.2.17]{voight}, the equation $x^2 = -p$ has a solution
$x_0\in \ZZ_\ell$. Hence $(x_0,0)$ is a solution
of (\ref{eq:norm_eq}) in $\ZZ_\ell^2$. Conversely, suppose
(\ref{eq:norm_eq}) has a solution $(x_0,y_0)\in 
\ZZ_\ell^2$. Reducing (\ref{eq:norm_eq}) modulo $\ell$ shows that
$x_0\amod{\ell}$ is a square root of $-p$ in $\mathbb{F}_\ell$, i.e.~$\legendre{-p}{\ell} = 1$.

For the rest of the proof we assume that $\ell=2$, which implies that
$4|D$. First, suppose that $(x, y)\in \ZZ_2^2$ is a solution of
(\ref{eq:norm_eq}). Since $x^2,\,y^2 \equiv 0,\,1 \mbox{ or } 4 \amod{8}$ and at least one
of $x,\,y$ lies in $\ZZ_2^\times$ because $p$ is odd, we see that the 
pair $(x^2, y^2)$ takes on five possibilities modulo $8$: 
\[(x^2,y^2) \equiv (0,1),\,(1,0),\,(1,1),\,(1,4) \mbox{ and }
  (4,1)\pmod{8}.\] Each possibility  puts the following respective constraint on
$p$ and $D$: 
        \begin{align*}
          -p + \frac{D}{4} &\equiv 0 \amod{8}, &  -p &\equiv 1
                                                     \amod{8}, &
                                                                 -p +
                                                                 \frac{D}{4}
          &\equiv 1\amod{8}, \\
                            -p + D &\equiv 1\amod{8}, &
          -p + \frac{D}{4} &\equiv 4\amod{8}. & & 
        \end{align*}
We have proved the necessity part of the lemma for the case $\ell=2$. 

Conversely, let us show that the above congruence conditions  are
also sufficient.  From the  discussion above, each of these conditions
guarantees the  existence of a solution $(\tilde{x}, \tilde{y}) $ of
equation (\ref{eq:norm_eq}) in $(\ZZ/8\ZZ)^2$ such that 
 either $\tilde{x}$ or
$\tilde{y}D/4$ lies in $(\ZZ/8\ZZ)^\times$. Now from a multivariate
version of Hensel's lemma \cite[Lemma~12.2.8]{voight}, the pair
$(\tilde{x}, \tilde{y})$  lifts
to a solution  of (\ref{eq:norm_eq}) in $\ZZ_2^2$. The
sufficiency is proved.

  Therefore, 
$-p \in N_{K/\QQ}(\calO_2^\times)$ if and only if one of the following
conditions holds:
    \begin{enumerate}
        \item[(a)] $p \equiv 7 \amod{8}$;
        \item[(b)] $-p + \frac{\disc(\calO)}{4} \equiv 0,\,1 $ or $4 \amod{8}$;
        \item[(c)] $-p + \disc(\calO)\equiv 1 \amod{8}$.
    \end{enumerate}
\end{proof}
\begin{proof}[Proof of Theorem~\ref{thm:main_1}]
  If $\calH_p\neq \emptyset$, then there is a regular action of
  $\Pic(\calO)[2]$ on $\calH_p$ by Proposition~\ref{prop:picO2}.  The
  criterion for the nonemptiness of $\calH_p$ follows from combining
  Proposition~\ref{prop:Sopt} with equation~(\ref{eq:13}) and
  Lemmas~\ref{lem:isom-B-p-infty}, \ref{lem:Sopt-nonempty} and
  \ref{lem:-p=norm}. 
\end{proof}
\bibliographystyle{amsalpha}
\bibliography{refs} 

\end{document}